\documentclass[12pt]{amsart}
\usepackage{graphicx}
\usepackage[active]{srcltx}
\usepackage{amsthm,mathrsfs}
\usepackage{a4wide}
\usepackage[OT2,T1]{fontenc}
\usepackage[russian,english]{babel}
\usepackage{amsfonts}
\usepackage{amsmath}
\usepackage{amssymb}
\usepackage{dsfont}
\newtheorem{lemma}{Lemma}
\newtheorem{theorem}{Theorem}
\newtheorem*{theorema}{Theorem A}
\newtheorem*{theoremb}{Theorem B}

\usepackage{float}
\usepackage{url}
\usepackage{enumerate}

\newcommand {\E} {\mathbb{E}}

\newcommand {\p} {\mathbb{P}}

\newcommand {\R} {\mathbb{R}}
\newcommand {\Q} {\mathbb{Q}}

\newcommand {\ve} {\varepsilon}

\makeatletter\def\blfootnote{\xdef\@thefnmark{}\@footnotetext}\makeatother
\allowdisplaybreaks
\title[On some questions of Arnold]{\bf On some questions of V.I. Arnold on the stochasticity of geometric and arithmetic progressions}

\author{Christoph Aistleitner} 
\address{Department of Mathematics and Statistics, Graduate School of Science, Kobe University, 1-1 Rokkodai, Nada-ku, Kobe 657-8501, Japan }
\email{aistleitner@math.tugraz.at}

\thanks{The author is supported by a Schr\"odinger scholarship of the Austrian Research
Foundation (FWF).}


\begin{document}

\begin{abstract}
In some of his final papers, V.I. Arnold studied pseudorandomness properties of finite deterministic sequences, which he measured in terms of their ``stochasticity parameter''. In the present paper we illustrate the background in probability theory and number theory of some of his considerations, and give answers to some of the questions raised in his papers.
\end{abstract}

\date{}
\maketitle

\nocite{a1}\nocite{a2}\nocite{a3}\nocite{a4}\nocite{a5}\nocite{a6}\nocite{a7}\nocite{a8}\nocite{a9}\nocite{a10}\nocite{a11}

\section{Introduction} \label{sec1}

In some of his final papers, V.I. Arnold investigated pseudorandomness properties of finite deterministic sequences of integers or reals. Amongst several other types of sequences, he in particular investigated arithmetic progressions, geometric progressions, continued fraction expansions, permutations and quadratic residues; see the papers~\cite{a1}--\cite{a11} in the bibliography below. To quantify the degree of pseudorandomness of these sequences, Arnold used a ``stochasticity parameter'' $\lambda_n$, and several of the mentioned papers of Arnold begin with a short history of the introduction of this stochasticity parameter in Kolmogorov's seminal ``Italian paper''~\cite{kolm}\footnote{An english translation of Kolmogorov's Italian paper, together with an introduction by M.A. Stephens, can be found in~\cite{break}.
}. In this paper, which was published in the same issue of an Italian actuarial journal as the papers of Glivenko~\cite{glivenko} and Cantelli~\cite{cant} on what we know today as the \emph{Glivenko--Cantelli theorem}, Kolmogorov proved that the normalized maximal deviation between the empirical distribution $C_n(X)$ of a set of $n$ independent random variables and the underlying distribution itself has an universal limiting distribution; this fact can be used to test whether a given sample is likely to be a realization of an independent, identically distributed (i.i.d.) random sequence drawn from a certain distribution. Kolmogorov's method also had a political dimension in the poisonous atmosphere of the time of its development; he himself used it to defend Mendelian genetics against the state-supported Lamarckism of Lysenko in an article~\cite{kolm2} in 1940, but had to publicly retract the article eight years later.\footnote{The situation of mathematical life under Stalin's rule is described in detail in G.G. Lorentz' paper on ``Mathematics and politics in the Soviet Union from 1928 to 1953''~\cite{lorentz}, and in the book ``Golden years of Moscow mathematics''~\cite{gold}, which contains a chapter on Kolmogorov, written by Arnold.}\\

Let $x_1, \dots, x_n$ be real numbers, sorted in increasing order. Their \emph{empirical counting function} $C_n(X)$ is defined as the number of elements $x_m$ which are not larger than $X$; that is, we have
\begin{equation} \label{cn}
C_n(X) = \left\{ \begin{array}{ll} 0 & \textrm{for $X < x_1$,}\\ m & \textrm{for $x_m \leq X < x_{m+1}$,}\\n & \textrm{for $x_n \leq X$.}\end{array} \right.
\end{equation}
In contrast, the \emph{theoretical counting function} $C_0(X)$ is given by 
$$
C_0(X) = n \p (x \leq X),
$$
that is by the expected number of values not exceeding $X$ of independent observations of the random variable $x$ (in other words, this is $n$ times the cumulative distribution function of $x$). Let 
$$
\mathcal{F}_n = \sup_X \left| C_n(X) - C_0(X)\right|.
$$
Then the \emph{stochasticity parameter}$\lambda_n$ is defined by
$$
\lambda_n = \frac{\mathcal{F}_n}{\sqrt{n}}.
$$
Kolmogorov proved, under the assumption that the cumulative distribution function of $x$ is continuous, that $\lambda_n$ has a limiting distribution $\Phi$ as $n \to \infty$, which is given by
\begin{equation} \label{kolm}
\Phi(X) = \lim_{n \to \infty} \p (\lambda_n \leq X) = \sum_{k=-\infty}^{\infty} (-1)^k e^{-2k^2 X^2}, \qquad \textrm{for $X > 0$.}
\end{equation}
Note that the distribution $\Phi$, which is now known as \emph{Kolmogorov distribution}, is \emph{universal} -- it does not depend on the initial distribution of $x$ (although it is assumed, as noted, that the initial distribution is continuous). Thus a given (large) sample of observations may be accepted as a realization of a sequence of i.i.d.~random variables having a specific distribution if and only if the value of its stochasticity parameter $\lambda_n$, calculated with respect to this distribution, is contained in an interval which contains the largest part of the mass of the Kolmogorov distribution. This principle if the basis of the \emph{Kolmogorov--Smirnov test} in statistics. A possible choice for such an interval may be $[0.4,1.8]$, since the Kolmogorov distribution assigns a probability of less than one per cent to the range outside of this interval.\\

Arnold used the stochasticity parameter to investigate the degree of randomness of finite deterministic sequences. In~\cite{a8}, the following examples are given: the sequence
$$
03, \quad 09, \quad 27, \quad 81, \quad 43, \quad 29, \quad 87, \quad 61, \quad 83, \quad 49, \quad 47, \quad 41, \quad 23, \quad, 69, \quad 07
$$
which is constructed as a geometric progression modulo 100, and the sequence
$$
37, \quad 74, \quad 11, \quad 48, \quad 85, \quad 22, \quad 59, \quad 96, \quad 33, \quad 70, \quad 07, \quad 44, \quad 81, \quad 18, \quad 55,
$$
which is an arithmetic progression modulo 100. The number of elements is in both cases $n=15$. For the first sequence Arnold obtained the stochasticity parameter $\lambda_{15} \approx 0.70$, while for the second he got $\lambda_{15} \approx 0.33$. This should be compared to the corresponding values for the Kolmogorov distribution, which are $\Phi(0.70)\approx 0.30$ and $\Phi(0.33) < 0.001$. This indicates that the first sequence is rather likely to be ``random'' than the second sequence.\\

Arnold acknowledges that this argument has methodological flaws. On the on hand, the sample size $n=15$ may be too small to assume that the stochasticity parameter of a random sequence already follows Kolmogorov's distribution. On the other hand, Kolmogorov's result is explicitly only applicable in the case of continuous distribution functions, while a distribution assigning positive probabilities only to the numbers $0,\dots,99$ (or any other set of integers) is of course discrete (for this issue, see Section~\ref{sect2} below). However, in Arnold's words, despite ``being more a method of natural sciences than a mathematical theorem'', such empirical observations still ``can provide useful information about the nature of the variable we are considering''. Based on a large number of empirical observations, Arnold for example observed that geometric progressions usually are much more ``random'' than arithmetic progressions, as in the case of the two sequences mentioned above.\\

In the papers~\cite{a1}--\cite{a11}, Arnold collected a large number of empirical observations, rigorous mathematical theorems and open problems concerning the ``randomness'' of deterministic sequences. The purpose of the present paper is to comment on some of the observations, illustrate the context of these investigations in probability theory and number theory, and to answer some particular problems.\\

It should be noted that Kolmogorov's stochasticity parameter is just one out of many possible ways to measure the randomness of a given sequence. Later in his life, Kolmogorov himself established a \emph{complexity theory}, which can be used to formalize randomness (see~\cite{lv}). More recently, an effort to measure the pseudorandomness properties of finite sequences was made by Mauduit and S{\'a}rk{\"o}zy, who introduced and studied several new measures of pseudorandomness (first for binary sequences, starting with~\cite{ms}, and later for sequences of more symbols~\cite{ahl,ahl2}). The problem is also discussed in detail in volume 2 of Knuth's ``The art of computer programming''~\cite{kn}.\\

Concerning Arnold's investigations, I think one should distinguish between several instances of the problem. Firstly, between the cases
\begin{enumerate}[1.]
\item The sample size $n$ being fixed, and
\item The sample size $n$ tending to infinity,
\end{enumerate}
and secondly between the cases
\begin{enumerate}[a)]
\item The underlying distribution being discrete, and
\item The underlying distribution being continuous.
\end{enumerate}

To see that the first distinction is necessary, we note that by the Chung--Smirnov law of the iterated logarithm (established by Chung~\cite{chung} and Smirnov~\cite{smirnov}; see also~\cite[p.~505]{sw}) for a sequence of independent, identically distributed random variables having a continuous distribution we have
\begin{equation} \label{chungs}
\limsup_{n \to \infty} \frac{ \lambda_n}{\sqrt{\log \log n}} = \frac{1}{\sqrt{2}} \qquad \textup{almost surely}.
\end{equation}
In other words, even if the typical value of $\lambda_n$ for fixed $n$ should be somewhere between $0.4$ and $1.8$, in the long run for an infinite sample of observations we should expect values of $\lambda_n$ as large as roughly $\sqrt{\log \log n}$ to occur from time to time (there even exist precise quantitative results how often such large values should be observed; see~\cite{strass}). Concerning the second distinction, one has to recall that Kolmogorov's limit theorem is only valid for continuous distributions; thus it is not a priori clear against which distribution the obtained value of $\lambda_n$ should be tested in the discrete case. This issue will be addressed in Section~\ref{sect3} below.\\

The outline of the remaining part of this paper is as follows. In the subsequent section, we will introduce the notion of the \emph{star-discrepancy}, which is a classical concept in analytic number theory. We will show that in the case of $\lambda_n$ being calculated with respect to the continuous uniform distribution on $[0,1]$, the star-discrepancy and Kolmogorov's stochasticity parameter coincide, and that consequently known results from discrepancy theory can be utilized to answer Arnold's questions. Section~\ref{sect2} explains the context of the Kolmogorov stochasticity parameter in empirical process theory, and shows what happens in the case when the underlying  distribution is discrete. In Section~\ref{sect3} we will discuss a conjecture of Arnold on the \emph{typical} value of the stochasticity parameter for sequences of the form $a^x A$ mod $N$, where $A$ is fixed and $x=1,2,\dots$. Here the word ``typical'' means that we want to obtain results which hold for almost all parameters $a$, in the sense of Lebesgue measure. In Section~\ref{sect4} we discuss the closely related problem asking for the typical value of the the stochasticity parameter of $a^x A$ mod $N$ where now $a>1$ is fixed, $x=1,2,\dots$, and $A$ is taken uniformly from $[0,N]$. In Section~\ref{sect5} we discuss the problem of arithmetic progressions with real (not necessarily rational) step size, which is closely connected with the theory of continued fractions. Finally, Section~\ref{sect6} contains the proof of a theorem stated in Section~\ref{sect4}.

\section{Uniform distribution modulo 1 and discrepancy theory} \label{secud}

Let $x_1,x_2,\dots$ be an infinite sequence of real numbers. This sequence is called \emph{uniformly distributed modulo one} (u.d.\ mod 1) if for all $X \in [0,1]$ the asymptotic relation
\begin{equation} \label{eq}
\lim_{n \to \infty} \frac{C_n(X)}{n} = X
\end{equation}
holds. Here $C_n$ is the empirical counting function of the sequence of fractional parts of $x_1,x_2,\dots$ (therefrom the name ``uniform distribution \emph{modulo one}''). In other words, an infinite sequence is u.d.\ mod 1 if every interval $[0,X]$ contains asymptotically the ``fair'' share of fractional parts of elements of the sequence, proportional to its length $X$. In a vague sense a sequence which is u.d.\ mod 1 can be interpreted as showing ``random'' behavior, since by the Glivenko-Cantelli theorem a sequence of i.i.d.~uniformly $[0,1]$-distributed random variables satisfies~\eqref{eq} almost surely.\\

The notion of uniform distribution modulo one originates (independently) in work of Bohl, Hardy--Littlewood, Sierpi\'nski and Weyl in the early years of the 20th century. The most important paper in the early theory of uniform distribution modulo one is certainly Weyl's~\cite{weyl} seminal paper of 1916. It contains, amongst many other important results, the \emph{Weyl criterion}, which states that a sequence $x_1,x_2,\dots$ is u.d.\ mod 1 if and only if for all $h \in \mathbb{Z} \backslash \{0\}$
$$
\lim_{n \to \infty} \frac{1}{n} \sum_{m=1}^n e^{2 \pi i h x_m} = 0,
$$
thereby linking the theory of uniform distribution with the theory of exponential sums and Fourier analysis. A detailed survey on the early years of uniform distribution theory is given in~\cite{hb} (in German).\\

The degree of uniformity of the distribution of a finite point set $x_1,\dots,x_n$ can be measured in terms of its \emph{star-discrepancy}, a notion which was established by Van der Corput in the 1930s. The star-discrepancy $D_n^*$ of points $x_1,\dots,x_n$ from the unit interval is defined as
$$
D_n^*(x_1,\dots,x_n) = \sup_{X \in [0,1]} \left| \frac{C_n(X)}{n} - X\right|,
$$
where again $C_n$ is the empirical counting function of the fractional parts of $x_1,x_2,\dots$. It is easily seen that an infinite sequence is u.d.\ mod 1 if and only if its star-discrepancy tends to zero as $n \to \infty$. Discrepancy theory is a heavily investigated subject, amongst other reasons because it has important applications in numerical mathematics. By Koksma's inequality the deviation between the integral of a function $f$ over the unit interval and the arithmetic mean of the function values $f(x_1),\dots,f(x_n)$ is bounded by the product of the variation of the function and the star-discrepancy of $x_1,\dots,x_n$. A similar inequality holds in the higher-dimensional setting, indicating that point sets having small discrepancy can be used for numerically approximating the integral of a function. This observation is the cornerstone of the \emph{Quasi-Monte Carlo method} (QMC method) for numerical integration. Since there exist points sets having a discrepancy of order almost $N^{-1}$, the convergence rate of QMC integration can be much faster than the asymptotic error rate of order $N^{-1/2}$ of so-called Monte Carlo integration, where random sampling points are used. The interested reader can find more information on uniform distribution theory and discrepancy theory in the monographs of Drmota--Tichy~\cite{dts} and Kuipers--Niederreiter~\cite{knu}. A comparison between Monte Carlo and Quasi-Monte Carlo methods can be found in the book of Lemieux~\cite{lem}.\\

It is easy to see that there is a close connection between the star-discrepancy $D_n^*$ and the stochasticity parameter $\lambda_n$ in the case when the sequence is contained in $[0,1]$ and the underlying distribution in the stochasticity parameter is assumed to be the uniform distribution on $[0,1]$. More precisely, in this case these two quantities coincide up to normalization, and we have
\begin{equation} \label{*}
\sqrt{n} D_n^*(x_1,\dots,x_n) = \lambda_n.
\end{equation}
The notion of the star-discrepancy can be clearly generalized to sequences on an arbitrary finite interval $[A,B]$ instead of $[0,1]$ (counting the points contained in the periodic continuation of a subinterval of $[A,B]$, and comparing to the normalized Lebesgue measure). Thus in the case of real sequences from a finite interval, which are compared to the uniform distribution on this interval, results from discrepancy theory can be directly translated into results for the stochasticity parameter. We will use this fact in Sections~\ref{sect3}-\ref{sect6} below. It should be noted that while the uniform distribution may be the ``natural'' choice to use for comparison with the empirical distribution of a deterministic set of real numbers, there also exist many classes of sequences of reals whose limit distribution is different from the uniform distribution; many examples can be found in the book of Strauch and Porubsk{\'y}~\cite{sp}.\\

The classical theory of the star-discrepancy does not apply to the case of sequences which only have a finite number of possible values, and whose empirical distribution is compared with a discrete distribution. In particular, the results mentioned in Sections~\ref{sect3}-\ref{sect6} below cannot explain Arnold's observations in this setting, such as the apparent difference in the degree of randomness between the residues of arithmetic and geometric progressions of integers.

\section{Applying the Kolmogorov distribution to discrete random variables} \label{sect2}

In Kolmogorov's theorem, the assumption that the i.i.d.~random variables under consideration have a continuous distribution is crucial. Consider, for example, the case of $n$ independent fair Bernoulli random variables $z_1,\dots,z_n$ (that is, each of them is either 0 or 1 with probability 1/2, respectively). It is easily seen that in this case the empirical counting function is given by
\begin{equation*}
C_n(X) = \left\{ \begin{array}{ll} 0 & \textrm{for $X < 0$,}\\ n - \sum_{m=1}^n z_m & \textrm{for $0 \leq X < 1$,}\\1 & \textrm{for $1 \leq X$.}\end{array} \right.
\end{equation*}
Consequently, we have
$$
\lambda_n = \frac{\left|\sum_{m=1}^n z_m - n/2\right|}{\sqrt{n}}.
$$
Thus, in this simple setting by the central limit theorem the limit distribution of $\lambda_n$ is the distribution of $|z|$, where $z$ is a normal random variable with expectation 0 and variance 1/4 (this distribution is called a \emph{half-normal distribution}). In particular, in this setting the limit distribution of $\lambda_n$ is \emph{not} the Kolmogorov distribution. Note that the half-normal distribution has properties which are totally different from those of the Kolmogorov distribution; for example, its mass is not separated from the origin, and thus (in contrast to the Kolmogorov distribution) it is not unlikely to observe values close to 0.\\

In this context, Arnold writes~\cite[p.~35]{a8}:
\begin{quote}
Kolmogorov proved his theorem for real random variables (with continuous distribution functions). Earlier I (unlawfully) applied Kolmogorov's universal distribution $\Phi$ to variables assuming only integer values or even just a finite number of values (remainders of division by an integer $N$).\\
I know, of course, that mathematical rigor does not allow us to make unsubstantiated generalizations like that. But as a natural scientist I believe that, e.g., results in astronomy should not depend on whether the distance measured in some units [\dots] takes real or just integer values.\\
Therefore I hope that one can apply the Kolmogorov theory not just to real random variables (with continuous distribution functions) but also to other variables; for example, this theory should have generalizations to random variables whose values are integers $x \in \mathbb{Z}$, or points on the circle $S^1$, or remainders $x \in \mathbb{Z}_N = \mathbb{Z} / N \mathbb{Z}$, or even rational numbers ($x \in \mathbb{Q}$).\\
Of course, all these generalized theories should be rigorously formulated and proofs should be given, I hope this will be done (by mathematicians of the future). However, even before that I will be using Kolmogorov's distribution~$\Phi$ in these more general cases (in the hope that it is a sufficiently good approximation to genuine distributions of the randomness parameter in these generalized theories).
\end{quote}

Actually, these tasks are not left to future mathematicians. The theory of empirical processes has been intensively investigated for several decades, and there exist results which are even much more general than those alluded to by Arnold. Let $x_1,x_2,\dots$ be i.i.d.~random variables. Using the definitions from above, we set
\begin{equation} \label{gdef}
G_n(X) = \frac{C_n(X) - C_0(X)}{\sqrt{n}}.
\end{equation}
This stochastic process is called the \emph{empirical process}, and we clearly have
$$
\lambda_n = \sup_X |G_n(X)|.
$$
Whenever we assume that $X$ is \emph{fixed} and let $n \to \infty$, then by the central limit theorem $G_n(X)$ converges to a normal random variable with mean 0 and variance $C_0(X) \left(1 - C_0(X)\right)$. However, much more is true. The sequence of random processes $(G_n)_{n=1,2,\dots}$ converges, in an appropriate sense, to a Gaussian limit process $G$. The convergence here is so-called \emph{weak convergence in the Skorokhod space}. Let $B$ denote the \emph{(standard) Brownian bridge} on $[0,1]$, which is, informally speaking, a (standard) Brownian motion under the additional condition that $B(1)=0$. If $W$ is a (standard) Brownian motion, which is also called a (standard) Wiener process, then a (standard) Brownian bridge on $[0,1]$ is given by
\begin{equation} \label{repres}
B(t) = W(t) - tW(1), \qquad t \in [0,1].
\end{equation}
Using this terminology, the limit process $G$ in the aforementioned limit theorem may be written as
$$
G(X) = B(C_0(X)), \qquad X \in \mathbb{R}.
$$
This limit theorem is called \emph{Donsker's theorem}. It can be found, together with definitions of weak convergence and of the Skorokhod space, and together with a detailed account on empirical processes, in the monographs of Shorack--Wellner~\cite{sw} or van der Vaart--Wellner~\cite{vdv}.\\

For the sake of shortness, I do not want to give a detailed definition of weak convergence. In a simplified view, we may understand that weak convergence means the convergence in distribution of all ``simple'' functionals of $G_n$ to the corresponding functionals of $G$. In the case of the stochasticity parameter $\lambda_n$ this functional is the $L^\infty$-norm, and by Donsker's theorem the distribution of $\lambda_n$ converges to the distribution of
\begin{equation} \label{limitd}
\sup_{X \in \mathbb{R}} \left|B(C_0(X))\right|.
\end{equation}
However, in the case when the function $C_0$ is continuous the distributions of
$$
\sup_{X \in \mathbb{R}} \left|B(C_0(X))\right| \qquad \textrm{and} \qquad \sup_{X \in [0,1]} \left|B(X)\right|
$$
are the same. Thus in this case the limit distribution of $\lambda_n$ does not depend on $C_0$, and is just the distribution of the supremum of the (standard) Brownian bridge -- which is Kolmogorov's distribution.\\

On the other hand, if $x_1,x_2,\dots$ are discrete i.i.d.~random variables having only a finite number of possible values, then clearly the distribution of~\eqref{limitd} (which is the limit distribution of $\lambda_n$) is the distribution of the maximum of the Brownian bridge at a finite number of positions. In particular, if $x_1,x_2,\dots$ have $N$ possible values and each has probability $1/N$, then the limit distribution of $\lambda_n$ is the same as the distribution of
\begin{equation} \label{distdist}
\max_{m=1,\dots,N-1} \left|B\left(\frac{m}{N}\right)\right|.
\end{equation}
Arnold's investigations are based on the conviction that the Kolmogorov distribution is a good approximation for the distribution of~\eqref{distdist}, provided $N$ is ``large''. It is not difficult to see that this actually is the case. For example, based on the representation~\eqref{repres} and on the fact that the distribution of the maximum of a Brownian motion is well-known (due to the so-called \emph{reflection principle}), one could quite easily find explicit upper bounds for the deviation between the distribution of~\eqref{distdist} and Kolmogorov's distribution.\\

The speed of convergence in Kolmogorov's limit theorem (and Donsker's theorem) has also been investigated; a fundamental results in this context is the Koml\'os--Major--Tusn\'ady theorem~\cite{kmt,kmt2}. Furthermore, the problem concerning the convergence of the empirical process has been generalized to far more general settings than that of random variables on $\mathbb{R}$ and test sets of the form $(-\infty, X]$, as in Kolmogorov's theorem and Donsker's theorem, and now covers empirical processes indexed by (general) sets as well as empirical processes indexed by functions. Roughly speaking, the convergence properties in this case depend on the \emph{entropy} of the class of test sets (or test functions, respectively). In particular, the instances mentioned by Arnold (such as points on a circle) are covered by this general theory. For more information on this topic, and for all the technical details, the reader is once again referred to~\cite{sw,vdv}.

\section{The stochasticity parameter of geometric progressions} \label{sect3}

In~\cite[p.~35]{a8}, Arnold mentions the following example:
\begin{quote}
{\bf Example.} Modulo $N$ remainders of $n$ terms 
$$
\{a^x A \quad (\textup{mod}~ N)\}, \qquad (n=0,1,\dots,(n-1)\}
$$
of the geometric progression with the first term $A$ and integer ratio $1 < a < N$ can look like a random sample of points uniformly distributed over $\mathbb{Z}_N$ provided that the number $n$ of terms is not ``too large''. For example, we can take $n \approx T/2$ or $\theta T$ with the constant $\theta$ separated from 0 and 1, $0 < \theta < 1$, where $T=T(N,a)$ is the period of the sequence [in the displayed formula above] consisting of remainders of terms of progression modulo $N$.\\
For different initial points $A$ (of the dynamical system $\mathbb{Z}_N \mapsto \mathbb{Z}_N$ sending $x$ to $ax$) the Kolmogorov stochasticity parameter $\lambda_n$ of the corresponding $n$ remainders of terms of the geometric progression takes different values $\lambda_n(A)$.\\
Computing all these numbers and counting their distribution (corresponding to the uniform distribution of the initial point $A$ in $\mathbb{Z}_N$) I have found (based on several hundreds of such experiments) a reasonable similarity of the distribution of obtained values $\lambda_n(A)$ of the Kolmogorov stochasticity parameter for different orbits of our dynamical system (i.e., for different values of $A$) with the universal distribution $\Phi$ of the stochasticity parameter~[\dots] \\
The similarity with the Kolmogorov distribution $\Phi$ is not a theorem but an empirical observation. In mathematical terms, it should be called a ``conjecture'' that as $N \to \infty$, the distribution of the values of the Kolmogorov parameter for $\mod N$ remainders of terms of $N$ geometric progressions (corresponding to $N$ initial points $A$) tends to $\Phi$.\\
The ``theoretical distribution'' $C_0$ of $\mod N$ remainders in these experiments was assumed to be uniform~[\dots] As far as I know, this conjecture about the uniform distribution of remainders is not yet rigorously proved$^1$, but in the book [3] (about Galois theory) I gave some ``physical proofs'' of (more general) theorems about uniform distribution (including the uniform distribution of fractional parts of numbers $a^x$ for almost all real bases $a$)~[\dots] These ``physical arguments'' are not proofs in the mathematical sense~[\dots]
\end{quote}

There are two misprints in the quoted text; on the right-hand side of the displayed formula at the beginning, $n$ should be $x$, and later in the text $x \mapsto ax$ evidently should be $x \mapsto a^x$. The footnote mentioned in the text is the following:
\begin{quote}
$^1$I am grateful to A.A. Karatsuba who brought to my attention the article by J.F. Koksma, \emph{Ein mengentheoretischer Satz \"uber die Gleichverteilung modulo Eins}, Compositio Math. {\bf 2} (1935), 250-258.
\end{quote}
Koksma's paper is cited as~\cite{koks} in the bibliography of the present paper. The main result in Koksma's paper is the fact that the sequence of fractional parts of $a^x,~x=1,2,\dots,$ is uniformly distributed modulo 1 (in the sense of uniform distribution theory, as introduced in Section~\ref{secud}) for almost all $a>1$. The book referred to in the quoted text is Arnold's book \emph{Dynamics, statistics and projective geometry of {G}alois fields}. He refers to the Russian version~\cite{a13} of 2005; in the meantime, an English translation~\cite{a14} has also been published.\\

In the discrete setting, Arnold's conjecture is probably extremely difficult. Some results in this direction have been proved; see for example~\cite{chang,shp}.\\

The situation in the continuous case, that is in the case of real values for the parameters $a$ and $A$, the situation is quite different from the discrete case, for a number of reasons. On the one hand, in the discrete case it does not make sense to keep $A$ and $a$ fixed and let $n \to \infty$, since the sequence $a^x A ~\textup{mod}~ N$ for $x =1,2,\dots$, is periodic (which implies that in this case for all $A$ and $a$ we have $\lambda_n \to 0$ as $n \to \infty$). This is different when $A$ and $a$ are real numbers, and it makes perfect sense in this case to ask for the behavior of $\lambda_n$ as $n \to \infty$. On the other hand, while in the case of real $A$ and $a$ the problem asking for the distribution of $a^x A ~\textup{mod}~ N$ is typically extremely complicated for \emph{fixed} values of $A$ and $a$, it is possible to obtain sharp results for \emph{typical} values of $A$ and $a$ (where ``typical'' should be understood in the sense of Lebesgue measure: the exceptional set has measure zero).\\

The fact that in the continuous case sharp results for ``typical'' sequences may be obtained is noted in the quotation from Arnold's paper, and is reflected in the reference to Koksma's paper. However, I could not find out what exactly Arnold refers to. He writes: ``in the book~[3] (about Galois theory) I gave some ``physical proofs'' of (more general) theorems about uniform distribution (including the uniform distribution of fractional parts of numbers $a^x$ for almost all real bases $a$)'', with a reference to the book listed as item~\cite{a13} in the bibliography of the present paper. However, actually no such results are contained in this book (I can only read the English translation~\cite{a14}, but it seems quite clear for me that it contains exactly the same material as the Russian original).\\

Of course it would be desirable to solve the problem of the distribution of $a^x A \quad (\textup{mod}~ N)$ for \emph{specific} values of $a$ and $A$, rather than only for almost all parameters. However, this is a notoriously difficult problem, and very little is known. For example, it is unknown whether the fractional parts of the sequences $(e^x)_{x\geq 1}$, $(\pi^x)_{x \geq 1}$, or $((3/2)^x)_{x \geq 1}$ are uniformly distributed modulo one or not. Actually, the situation is much worse. For example, we do not even know whether 
$$
\textup{limpsup}_{x \to \infty} \left\{ (3/2)^x \right\} - \textup{liminf}_{x \to \infty} \left\{ (3/2)^x \right\} > \frac{1}{2}
$$
(this is Vijayaraghavan's~\cite{viya} problem of 1940; in this statement and in the sequel $\{ \cdot \}$ denotes the fractional part of a real number). We also don't know whether or not there exists an $A \neq 0$ such that 
$$
\left\{A (3/2)^x \right\}  \in \left[0,\frac{1}{2} \right) \qquad \textrm{for all $x=1,2,\dots$} 
$$
(this is Mahler's~\cite{mahl} problem of 1968). Results for this kind of problem are very scarce; see for example~\cite{aki,dub} for recent contributions. Another confirmation of how meager our knowledge on these topics is, is the fact that although by Koksma's result for almost all $a$ the sequence of fractional parts of $a^x, ~x=1,2,\dots$ is uniformly distributed modulo one, we do not know even a single specific number $a$ which has this property.\\

In~\cite[p.~36]{a8}, Arnold formulates the following conjecture:
\begin{quote}
{\bf Conjecture}. The Kolmogorov stochasticity parameter $\lambda_n$ of residues modulo $N$ of $n$ terms of a geometric progression with an arbitrary ratio $a>1$ does not tend to zero as $n \to \infty$ (for almost all $a$, so that exceptional values form a set of Lebesgue measure zero on the real line).  
\end{quote}

The solution to this conjecture is known; the answer is affirmative. Arnold's conjecture, asserting that the Kolmogorov parameter $\lambda_n$ of a ``typical'' geometric progression is not too small, should be compared to the case of arithmetic progressions, where the Kolmogorov parameter of a typical sequence actually is too small (it tends to 0 as $n \to \infty$); see Section~\ref{sect5} below. However, in comparison with~\eqref{chungs} the assertion that $\lambda_n$ does not tend to 0 as $n \to \infty$ is too weak to capture the behavior of the Kolmogorov stochasticity parameter for a typical i.i.d.~random sequence. Under the supposition that a typical geometric progression behaves similar to a typical realization of an i.i.d.~random sequence, one could actually conjecture that even $\sqrt{\log \log n} \lambda_n$ does not tend to 0 as $n \to \infty$ for almost all $a>1$. As the following results from~\cite{a1} shows this stronger statement is also true, and the Kolmogorov stochasticity parameter for typical geometric progressions satisfies the Chung--Smirnov law of the iterated logarithm in exactly the same way as an i.i.d.~random sequence.

\begin{theorema}
Let $A>0$ and $N>0$ be fixed real numbers. Then for the sequence of remainders $a A, a^2 A, a^3 A, \dots$ modulo $N$ we have
$$
\limsup_{n \to \infty} \frac{\lambda_n}{\sqrt{\log \log n}} = \frac{1}{\sqrt{2}} \qquad \textrm{for almost all $a \in \R,~a > 1$},
$$
where the Kolmogorov stochasticity parameter $\lambda_n$ is calculated with respect to the uniform distribution on $[0,N]$.
\end{theorema}

This theorem is stated in~\cite{a1} only for the case of the fractional part of a sequence, that is for the case of $a^1 A, a^2 A, \dots$ being reduced modulo 1. However, it is easily seen that by a simple change of scale the theorem also covers the case of $a^1 A, a^2 A, \dots$ being reduced modulo $N$, by means of replacing $A$ by $NA$. Thus, the answer to Arnold's conjecture is affirmative.

\section{The stochasticity parameter of lacunary sequences} \label{sect4}

In the previous section we discussed the problem whether or not a geometric progression is uniformly distributed modulo 1 or not. A quite similar problem to that of deciding for which $a$ the fractional parts of $a^x A,~x = 1,2,\dots$ (for fixed $A$) are uniformly distributed modulo one is that of deciding for which $A$ the sequence of fractional parts of $a^x A,~x=1,2,\dots$ (for fixed $a$) is u.d.\ mod 1 -- that is, in the case of integral $a$, the problem of deciding whether $A$ is a so-called \emph{normal number} in base $a$ or not. The property of being a normal number in a certain base is usually defined in terms of the number of occurrences of digits and blocks of digits in the digital expansion of the number; for example, a number $A$ is normal in base 10 if in its decimal expansion (after the decimal point) each possible digit 0,1,\dots,9 occurs with asymptotic frequency 1/10, each block of 2 digits such as 00, 01, etc. appears with asymptotic frequency 1/100, each block of 3 digits appears with asymptotic frequency 1/1000, and so on. It is not difficult to see that this property can be described in terms of the uniform distribution modulo 1 of $10^x A$; to see that is the case, one just has to note that the map $A \mapsto 10 A$ modulo 1 represents a shift to the left of the decimal digits of $A$, and that consequently counting the number of occurrences of certain digits is the same as summing the values of indicator functions of appropriate intervals at positions $A,\{10 A\},\{10^2 A\}$, etc. For example, the number of occurrences of the digit ``4'' among the first $n$ decimal digits (after the decimal point) of a number $A$ is equal to
$$
\sum_{m=0}^{n-1} \mathds{1}_{[0.4,0.5)}(\{10^m A\}),
$$
and in the same way we can count the number of occurrences of blocks of digits. By a famous result of Borel~\cite{borel}, almost all numbers are normal (in every given integer base). Constructing normal numbers is possible, but rather difficult. However, deciding whether a number such as for example $\pi$, $e$, $\sqrt{2}$ is normal in a given base or not is an extremely difficult problem, and is entirely open. For example, it is often conjectured that all algebraic irrationals are normal (in every integer base), but we are very, very far from proving such a result (see~\cite{bail} for the state of research).\\

Borel's result is the first appearance of what we now call the \emph{strong law of large numbers}, in the special case of the so-called Rademacher functions (which form, as later observed by Steinhaus, a sequence of i.i.d.~random variables). Formulated in base 10, Borel's result states that the sequence of fractional parts of $10^x,~x=1,2,\dots$, is u.d.\ mod 1; this is a special case of the by now well-established principle that so-called lacunary sequences of functions exhibit properties which are typical for sequences of independent random variables. Here ``lacunary sequence of functions'' means a sequence of the form $f(a_1 y),f(a_2 y),f(a_3 y),\dots$, where $f$ is a function which is periodic with period 1 and $a_1,a_2,\dots$ is a quickly increasing sequence of integers, satisfying for example the \emph{Hadamard gap condition} $a_{x+1}/a_x \geq q > 1,~x=1,2,\dots$ (in our case, the role of the 1-periodic function is played by the fractional part function $f(y)=\{y\}$). Questions concerning the behavior of such function systems for almost all $y$ can be handled in the same way as questions concerning the almost sure behavior of systems of i.i.d.~random variables - this is the reason why many probabilistic results for lacunary sequences are known; see~\cite{kac} for a classical and~\cite{ab} for a recent survey. \\

In~\cite[p.~36]{a8}, following the conjecture mentioned in the previous section, Arnold formulates the following conjecture:
\begin{quote}
Moreover, one can conjecture that for almost any base $a>1$ the following more general statement holds: The distribution of the values $\lambda_n(A)$ of the Kolmogorov stochasticity parameter $\lambda_n$ of the sequences of $n$ remainders modulo $N$ of geometric progressions starting at different points $A$ ($0 < A < N)$, tend[s], as $n \to \infty$, to the universal Kolmogorov distribution~$\Phi$ (under the assumption that the starting point $A$ is uniformly distributed on the interval $0 < A < N$). 
\end{quote}

Note that this conjecture is much stronger than the conjecture from the previous section, where it was only required that $\lambda_n$ does not tend to 0 as $n \to \infty$. However, there is also a difference between the probabilistic model which is used to specify a class of parametric sequences. In the previous section, the sequence $a^x A,~x =1,2,\dots$ was obtained by assuming $A$ to be fixed and allowing different values for the parameter $a$. In the present case, $a$ is fixed and $A$ is variable. Thus to solve the problem from the previous section (and in the case of reduction modulo 1) it was, roughly speaking, necessary to show that the functions $\{a A\},\{a^2 A\},\{a^3 A\},\dots$, understood as functions of $a$, show a behavior which is similar to that of sequences of i.i.d.~random variables. In the present case it has to be shown that the same functions, now understood as functions of $A$, also behave like i.i.d.~random variables. These two problems are technically quite different, and require different methods. Generally speaking, the case of lacunary sequences (that is, of assuming that $A$ is the variable and $a$ is fixed, as in the present section) is the case which has a longer research history, is better understood, and is easier to handle.\\

The asymptotic behavior of the Kolmogorov stochasticity parameter (or, in other words: the star-discrepancy) of lacunary sequences is an intensively studied subject. It turns out that precise results depend on fine number-theoretic properties of the growth factor $a>1$ in a very sensitive way. Quite recently, Fukuyama~\cite{fuk1} proved the following. 

\begin{theoremb}
The Kolmogorov stochasticity parameter $\lambda_n$ of the sequence $aA,a^2 A,a^3 A,\dots$ modulo 1 satisfies, for almost all $A \in [0,1]$, the asymptotic relation
\begin{eqnarray*}
\limsup_{n \to \infty} \frac{\lambda_n}{\sqrt{\log \log n}} = \left\{ \begin{array}{ll} \frac{\sqrt{84}}{9} & \textrm{for $a=2$,}\\ \frac{\sqrt{(a+1)a(a-2)}}{\sqrt{2 (a-1)^3}} & \textrm{if $a \geq 4$ is an even integer,} \\
\frac{\sqrt{a+1}}{\sqrt{2 (a-1)}} & \textrm{if $a \geq 3$ is an odd integer,} \\ \frac{1}{\sqrt{2}} & \textrm{if $a>1$ and $a^x \not\in \Q$ for $x=1,2,\dots$.} \end{array} \right.
\end{eqnarray*}
\end{theoremb}

The last case is particularly interesting; it covers the case when $a$ is a transcendental number. Since almost all numbers are transcendental, this is the typical case with respect to Lebesgue measure, and as in Section~\ref{sect3} there is a perfect accordance with the Chung--Smirnov LIL~\eqref{chungs} for i.i.d.~random variables.\\

A corresponding limit theorem for the distribution of $\lambda_n$ has not been proved so far; we state it below as a theorem.

\begin{theorem} \label{th1}
Let $a>1$ be a fixed real number for which $a^x \not\in \Q$ for $x=1,2,\dots$, and let $N>0$ also be fixed. Then for the Kolmogorov stochasticity parameter $\lambda_n$ of the sequence $a A, a^2 A, a^3 A,\dots$ mod $N$ we have
$$
 \lim_{n \to \infty} \p (A \in [0,N]:~\lambda_n \leq X) = \Phi(X) \qquad \textrm{for all $X \in \R$,}
$$
where $\p$ denotes the normalized Lebesgue measure on $[0,N]$, where $\Phi$ is the distribution function of the Kolmogorov distribution as defined in~\eqref{kolm}, and where $\lambda_n$ is calculated with respect to the uniform distribution on $[0,N]$.
\end{theorem}

Note that, as in Fukuyama's theorem above, the set of real numbers $a$ for which  $a^x \not\in \Q$ for $x=1,2,\dots$ has full Lebesgue measure. Thus Theorem~\ref{th1} proves Arnold's conjecture. The proof of Theorem~\ref{th1} will be given in Section~\ref{sect6}, at the end of this paper. If the assumption $a^x \not\in \Q$ for $x=1,2,\dots$ in the statement of Theorem~\ref{th1} is replaced by $a^x \in \Q$ for some positive integer $x$, then there still exists a limit distribution of the Kolmogorov stochasticity parameter $\lambda_n$. However, in this case the limit distribution depends on number-theoretic properties of $a$ and $x$ in a very complicated way, and is different from Kolmogorov's distribution.

\section{The stochasticity parameter of arithmetic progressions} \label{sect5}

In~\cite{a5}, Arnold proves two theorems on the stochasticity parameter of arithmetic progressions:\footnote{as noted in~\cite{a5}, by suitably choosing the scale the general case of arithmetic progressions modulo $N$ can be reduced to the case of arithmetic progressions modulo 1, that is to the case of fractional parts of arithmetic progressions.}
\begin{quote}
\begin{itemize}
\item For arithmetic progressions of fractional parts whose step $k$ is a rational number the Kolmogorov stochasticity parameter $\lambda_n$ tends to $0$ as $n \to \infty$ (indicating an asymptotic loss of randomness for such a long progression).
\item Contrary to the case of rational $k$, [there exist examples] in which the Kolmogorov parameter $\lambda_n$ does not tend to 0 as $n \to \infty$. It can even attain, though infrequently, arbitrarily large values (which cannot, however, exceed $\sqrt{n}$) for some sufficiently large lengths $n$ of the progressions. 
\end{itemize}
\end{quote}
The first result is proved using a relatively simple counting argument. The second result is proved constructively by giving an example of a value of $k$, specified in terms of its continued fraction expansion, which has the desired property. At the end of~\cite{a5}, Arnold writes:
\begin{quote}
I do not know whether the value of the Kolmogorov stochasticity parameter $\lambda_n$ of an arithmetic progression of fractional parts of the $n$ numbers $kx$ tends
to zero for almost all real numbers k, or whether it is just as often unbounded (it might also be ``generically'' bounded away from 0 and $\infty$). The ergodicity of the
Gauss--Kuzmin dynamical system $z \mapsto \{1/z\}$ suggests that any such asymptotic behavior of the stochasticity parameter $\lambda_n$ should have probability either 0 or 1 (in the space of values of the parameter $k$) (provided that it depends only on the asymptotic behavior of the partial quotients $a_s$ of the continued fraction of $k$ as
$s \to \infty$). But I do not know whether the probability is 0 or 1 for the types of behavior described above for the stochasticity parameter.
\end{quote}

In the later paper~\cite{a8} Arnold writes in this context:
\begin{quote}
Unfortunately, I don't know which alternative (``almost always'' or ``almost never'') holds for the properties formulated above: whether remainders of almost all arithmetic
progressions are random or nonrandom as far as the behavior of the values $\lambda_n$ of the stochasticity parameter of the first n elements of the sequence is concerned.
This general question is difficult to check both theoretically and experimentally: an empirical study of the fractional parts of arithmetic progressions presumably requires answering nontrivial questions about the statistics of continuous fractions, and the standard ``Gauss--Kuzmin'' statistics describing the distribution of incomplete continuous fractions of random real numbers (and their finite combinations) is insufficient to solve the above nontrivial problems.
\end{quote}

It must be noted that there is a significant difference (which is somewhat concealed in~\cite{a5}) between the two results cited above. Remember that the Kolmogorov stochasticity parameter $\lambda_n$ depends on the theoretical counting function $C_0(X)$ to which the empirical counting function $C_n(X)$ is compared. In the first of the two results from above, if the rational step size is $k=p/q$ for coprime $p,q$, then the theoretical counting function $C_0(X)$ is chosen as $n$ times the distribution function of the discrete uniform distribution on $\{0,1/q,\dots,(q-1)/q\}$. Of course this makes perfect sense, since the possible values of $\{kx\},~x=1,2,\dots$ are exactly the numbers $\{0,1/q,\dots,(q-1)/q\}$. On the other hand, in the second case (the case of irrational $k$) the theoretical counting function $C_0(X)$ is chosen as $n$ times the continuous uniform distribution on $[0,1]$. This also makes sense: by the equidistribution theorem of Bohl, Sierpi\'nski and Weyl the sequence $\{kx\},~x=1,2,\dots$ is uniformly distributed modulo 1 (in the sense of Section~\ref{secud}) for all irrational $k$, and thus in particular for almost all $k$ in the sense of Lebesgue measure. Consequently, the only reasonable choice for the theoretical counting function in the case of typical real $k$ is the continuous uniform distribution on $[0,1]$. Note that in this case, as mentioned in Section~\ref{secud}, the notion of the Kolmogorov stochasticity parameter coincides (up to normalization) with the star-discrepancy.\\

Arnold's observation that the problem of the stochasticity parameter (or, in the language of Section~\ref{secud}: the star-discrepancy) of a sequence of fractional parts $\{k x\},~x=1,2,\dots$, is intimately connected with the continued fraction expansion of the step $k$ is absolutely right. This observation was made independently by several mathematicians around 1920, such as Hecke, Ostrowski, Hardy--Littlewood, and Behnke. Roughly speaking, the smaller the continued fraction coefficients of $k$ are, the smaller the discrepancy of $\{k x\},~x=1,2,\dots$, is. There also exist many precise quantitative results giving discrepancy bounds in terms of the continued fraction coefficients of $k$; such results are presented in great detail in~\cite[Chapter 2, Section 3]{knu} and~\cite[Section 1.4.1]{dts}. Together with the profound results of Khintchine~\cite{khin,khin2} on the metric theory of continuous fractions one obtains the following result (\cite[Theorem 1.72]{dts}):\\

Suppose that $\psi(n)$ is a positive increasing function. Then
$$
n D_n^*(\{k\}, \dots, \{kn\}) = \mathcal{O} \left((\log n)\psi(\log \log n)\right) \qquad \textrm{as $n \to \infty$}
$$
for almost all $k \in \R$ if and only if
$$
\sum_{n=1}^\infty \frac{1}{\psi(n)} < \infty.
$$
In particular, this implies that for arbitrary $\ve>0$ we have
$$
D_n^*(\{k\}, \dots, \{kn\}) = \mathcal{O} \left(\frac{(\log n)(\log \log n)^{1+\ve}}{n} \right) \qquad \textrm{as $n \to \infty$}
$$
for almost all $k \in \R$. Consequently, by~\eqref{*}, we also have
$$
\lambda_n \to 0 \qquad \textrm{for almost all $k$,}
$$
which provides the solution of Arnold's problem.

\section{Proof of Theorem~\ref{th1}} \label{sect6}

It is easy to see that the value of the Kolmogorov stochasticity parameter $\lambda_n$ for testing the distribution of $aA,a^2A,a^3A,\dots,a^N A$ mod $N$ against the uniform distribution on $[0,N]$ is the same as the value of $\lambda_n$ when testing $a A/N,a^2A/N,a^3A/N,\dots,a^N A/N$ mod $1$ against the uniform distribution on $[0,1]$. Thus, for the proof of Theorem~\ref{th1} we may assume without loss of generality that $N=1$, which means that $A$ is taken uniformly from $[0,1]$ and the sequence we consider is the sequence of fractional parts $\{aA\},\{a^2A\},\{a^3 A\}, \dots$.\\

Our proof of Theorem~\ref{th1} follows the one given in~\cite{ab2} for the case of quickly increasing integer sequences $a_x,~x=1,2,\dots$, and which we adopt to the sequence $a^x,~1,2,\dots$ for real $a$ instead. All the necessary definitions and basic concepts (\emph{c\`{a}dl\`{a}g-function, Skorokhod space, Brownian bridge, tightness, weak convergence, \dots}) can be found for example in~\cite{bill}. The key ingredient is the following result of Fukuyama~\cite{fuk2}. It is stated in~\cite{fuk2} in a much more general multi-dimensional form, but we only need a special case of the one-dimensional version.

\begin{lemma}[{\cite[Theorem 1]{fuk2}}] \label{lemma1}
Let $f(y)$ be a measurable function which is of bounded variation on $[0,1]$ and satisfies
$$
f(y+1)=f(y), \qquad \int_0^1 f(y)~dy = 0, \qquad \int_0^1 f^2(y)~dy = 1.
$$
Then for all $X \in \R$ we have
$$
\lim_{n \to \infty} \p \left( A \in [0,1]: \frac{1}{\sqrt{n}} \sum_{x=1}^n f(a^x A) \leq X \right) = \frac{1}{\sqrt{2 \pi}} \int_{-\infty}^X e^{-y^2/2} ~dy,
$$
where $\p$ denotes the Lebesgue measure on $[0,1]$.
\end{lemma}

\begin{proof}[Proof of Theorem~\ref{th1}]
Let a number $a$ satisfying the assumptions of the theorem be given. As noted above, we may assume without loss of generality that $N=1$. As in~\eqref{gdef}, we define the empirical process $G_n$ by
$$
G_n(t) = \frac{\sum_{x=1}^n \mathds{1}_{[0,t]} (\{a^x A\}) - nt}{\sqrt{n}}, \qquad t \in [0,1].
$$
For each $n$, the paths of the process $G_n$ are c\`{a}dl\`{a}g-functions, and consequently $G_n$ is a stochastic process on the Skorokhod space $D[0,1]$. We want to show that $(G_n)_{n \geq 1}$ converges weakly to a standard Brownian bridge process $B(t)$, which is a Gaussian process having (almost surely) continuous paths, mean zero and covariance function $\textup{Cov}(B(t_1),B(t_2))=\E(B(t_1)B(t_2)) = t_1(1-t_2)$ for $t_1<t_2$ (see Section~\ref{sect2}).\\

To prove weak convergence $G_n \Rightarrow B$, by~\cite[Theorem 13.1]{bill} we have to show that all finite-dimensional distributions of $G_n$ converge to the corresponding finite-dimensional distributions of $B$, and that the sequence of processes $G_n(t), n =1,2,\dots$ is tight. By the well-known Cram\'er--Wold device (see for example~\cite[p.~343]{ash}), for the convergence of all finite-dimensional distributions of $G_n$ to those of $B$ it is sufficient to show that
\begin{equation} \label{distribc}
b_1 G_n(t_1) + \dots + b_m G_n(t_m) \xrightarrow{D} b_1 B(t_1) + \dots + b_m B(t_m)
\end{equation}
for all $m \geq 1$ and all $(b_1,\dots,b_m) \in \R^m,~0 \leq t_1 < \dots < t_m \leq 1$. Here ``$\xrightarrow{D}$'' denotes convergence in distribution. Thus, let $(b_1,\dots,b_m) \in \R^m$ and $0 \leq t_1 < \dots < t_m \leq 1$ be given. For $t \in [0,1]$, let $\mathbf{I}_{[0,t]}(y)$ denote the function $\mathds{1}_{[0,t]}(\{y\}) - t$; in other words, $\mathbf{I}_{[0,t]}$ is the indicator function of $[0,t]$, centered at expectation and extended with period 1. Then we have
$$
G_n(t) = \frac{\sum_{x=1}^n \mathbf{I}_{[0,t]} (a^x A)}{\sqrt{n}},
$$
and consequently
\begin{equation} \label{gsumme}
b_1 G_n(t_1) + \dots + b_m G_n(t_m) = \frac{1}{\sqrt{n}} \sum_{x=1}^n \sum_{k=1}^m b_k \mathbf{I}_{[0,t_k]} (a^x A).
\end{equation}
The function 
$$
\sum_{k=1}^m b_k \mathbf{I}_{[0,t_k]} (y)
$$
has integral zero (on $[0,1]$) and is periodic with period 1. Furthermore, some simple calculations show that we have
\begin{equation} \label{vari}
\int_0^1 \left(\sum_{k=1}^m b_k \mathbf{I}_{[0,t_k]}(y)\right)^2 dy = \sum_{k=1}^m b_k^2 t_k(1-t_k) + 2 \sum_{1 \leq k_1 < k_2 \leq m} b_{k_1} b_{k_2} t_{k_1} (1-t_{k_2}).
\end{equation}
Thus, by Lemma~\ref{lemma1}, the distribution of~\eqref{gsumme} converges to a normal distribution with mean zero and variance given by the right-hand side of~\eqref{vari}. On the other hand, using the covariance structure of the Brownian bridge, we can easily show that
$$
\E \left( b_1 B(t_1) + \dots + b_m B(t_m) \right)^2 = \sum_{k=1}^m b_k^2 t_k(1-t_k) + 2 \sum_{1 \leq k_1 < k_2 \leq m} b_{k_1} b_{k_2} t_{k_1} (1-t_{k_2}).
$$
Thus the distribution of the expression on the right-hand side of~\eqref{distribc} is also the normal distribution with mean zero and variance given by the right-hand side of ~\eqref{vari}. In other words, we have established~\eqref{distribc}, which proves that the finite-dimensional distributions of $G_n$ converge to those of $B$.\\

To prove that the sequence $G_n(t), ~n =1,2,\dots$ of processes is tight, we have to establish the two conditions required in~\cite[Theorem 13.2]{bill}. Both can be easily shown using the exponential inequalities and the dyadic chaining method of~\cite{phil} (which are stated there for the case of lacunary sequences of integers, but, as noted in the proof of Lemma (3.4) of~\cite{berk}, remain valid in the real case). It is well known that the functional $f \mapsto \sup_{0 \leq t \leq 1} |f(t)|$ is a continuous functional on $D[0,1]$. Thus by the continuous mapping theorem (see for example~\cite[Theorem 1.3.6]{vdv}), and since we have already established $G_n \Rightarrow B$, the distribution of $\sup_{0 \leq t \leq 1} |G_n(t)|$ converges to the distribution of $\sup_{0 \leq t \leq 1} |B(t)|$. However, since $\lambda_n = \sup_{0 \leq t \leq 1} |G_n(t)|$, and since the distribution of the maximum of the standard Brownian bridge is the Kolmogorov distribution, this proves the theorem.
\end{proof}

\section*{Acknowledgments}

I want to thank Katusi Fukuyama for his support during my one-year stay at Kobe University, and for his remarks concerning this manuscript. Many thanks also to the administrative staff at the Department of Mathematics of Kobe University.

%\bibliography{Arnoldbib}
%\bibliographystyle{abbrv}

\def\cprime{$'$} \def\cprime{$'$} \def\cprime{$'$}

\end{document}